\documentclass[12pt,reqno]{amsproc}
\usepackage{amsthm,wrapfig,amsmath,amssymb,amsfonts,setspace,verbatim,graphicx,mathtools,mathrsfs,commath,float,mdframed,frame,xcolor,qtree, ulem,afterpage,colortbl}
\usepackage[hidelinks]{hyperref}

\usepackage[shortlabels]{enumitem}

\usepackage[T1]{fontenc}

\DeclareMathOperator{\dom}{dom}

\DeclareMathOperator{\diam}{diam}

\usepackage{newunicodechar}
\usepackage{lmodern}
\usepackage[a4paper,margin=1.1in]{geometry}
\newtheorem{ut}{Theorem}

\newtheorem{up}{Proposition}

\newtheorem{obs}{Observation}

\newtheorem{uc}[ut]{Corollary}
\newtheorem{ucl}{Claim}

\theoremstyle{remark}

\definecolor{reddy}{HTML}{EA1174}
\definecolor{greenie}{HTML}{73F69C}

\theoremstyle{definition}

\newtheorem{ur}{Remark}

\begin{document}

\title[Exponential iteration and Borel sets]{Exponential iteration and Borel sets}

\address{Department of Mathematics and Data Science, College of Coastal Georgia, Brunswick GA 31520, United States of America}
\email{dsl0003@auburn.edu,dlipham@ccga.edu}
\author[D.S. Lipham]{David S. Lipham}

\begin{abstract}We determine the exact Borel class of  escaping sets  in the exponential family $\exp(z)+a$. We also prove that the sets of non-escaping Julia points for many of these functions are topologically equivalent. \end{abstract}

\maketitle

\section{Introduction}

The escaping set $$I(f)=\{z\in \mathbb C:f^n(z)\to\infty\}$$ is one of the most studied objects in complex dynamics  \cite{ere0,ere,sch,rem2}. For any entire function $f$, it is easily seen to be an $F_{\sigma\delta}$-subset of the complex plane $\mathbb C$.  However, it was recently shown by Lasse Rempe  \cite{rem3} that if $f$ is transcendental (with an essential singularity at $\infty$), then $I(f)$ is not $F_{\sigma}$. Here we will strengthen this result for  functions in the exponential class $$f_a(z)=e^z+a\;;  \hspace{1em}a\in \mathbb C.$$ The main result of the paper is the following.

\begin{ut}\label{t1} $I(f_a)$ is not $G_{\delta\sigma}$ for any $a\in \mathbb C$. \end{ut}

{    In other words, $I(f_a)$ cannot be written as a countable union of $G_{\delta}$-subsets of $\mathbb C$.    Theorem 1 is new even for the plain exponential $f_0(z)=e^z$, and  completes the Borel classification of $I(f_a)$.}

\begin{figure}[h] 
\begin{tabular}{| c | c  | c  | c}\hline
   \cellcolor{gray!20}$G_\delta$ \cite[Theorem 4.1]{lip}& \cellcolor{gray!20}$G_{\delta\sigma}$ (Theorem 1)& \cellcolor{greenie!45}$G_{\delta\sigma\delta}$ & \cellcolor{greenie!45}$\ldots$  \\ \hline
   \cellcolor{gray!20}$F_{\sigma}$ \cite[Theorem 1.2]{rem3} & \cellcolor{greenie!45}$F_{\sigma\delta}$ \cite[Section 1]{rem3} & \cellcolor{greenie!45}$F_{\sigma\delta\sigma}$ & \cellcolor{greenie!45}$\ldots$ \\ \hline
  \end{tabular}
 \caption{Borel classes of $I(f_a)$ (in green only).}
\end{figure}

{    A simple consequence of Theorem \ref{t1} is that given any $G_{\delta}$-set of escaping points, there exists an escaping point whose orbit does not enter the set. \begin{uc}If $X$ is a $G_{\delta}$-subset of $\mathbb C$ that is contained in  $I(f_a)$, then the pre-images $f_a^{-n}[X]$ do not cover $I(f_a)$. Thus there exists $z\in I(f_a)$ such that $f^n_a(z)\notin X$ for all $n=0,1,2,3\ldots$.\end{uc}}

Next we focus on the  $f_a$'s  which have attracting or parabolic cycles. The parameters $a$ associated with this class form a large and conjecturally dense subset of $\mathbb C$. The conjugacy in \cite{rem2} established that all  escaping sets are mutually homeomorphic in this context. Using a localized version of Theorem \ref{t1}, we will prove the following complementary result wherein  $J(f_a)$ denotes the Julia set of $f_a$.

\begin{ut}\label{t2}If $f_a$ and $f_b$ have attracting or parabolic cycles, then $J(f_a)\setminus I(f_a)$  and $J(f_b)\setminus I(f_b)$ are topologically equivalent (homeomorphic).\end{ut}

The homeomorphism in Theorem \ref{t2} is not necessarily induced by a homeomorphism of the escaping sets, as $J(f_a)$ and $J(f_b)$ are often non-homeomorphic. {  The result also cannot be extended to all $a\in \mathbb C$; the spaces in Theorem \ref{t2} are totally disconnected \cite[Corollary 10]{lip1}, but there are other parameters (e.g.\ postsingularly finite)  for which $J(f_a)\setminus I(f_a)$ contains  unbounded connected sets \cite[Section 2]{rem}. }

\subsection*{Ideas behind the proofs}We will prove Theorem \ref{t1} by constructing a stratification, or `tree', of first category $G_{\delta}$-subspaces of $I(f_{-1})$, such that every infinite branch of the tree has an accumulation point in $I(f_{-1})$. This idea comes from a classical proof that the infinite power of the rationals $\mathbb Q ^\omega$ is not $G_{\delta\sigma}$. The stratification sets in that proof are of the form $\{q_0\}\times\ldots \times\{q_n\} \times \mathbb Q \times \mathbb Q\times \ldots$, while ours will be defined with rates of escape in mind. Results from  \cite{rem2} will allow us to work inside a relatively simple topological model of $J(f_{-1})$, and to generalize from $I(f_{-1})$ to $I(f_a)$. For attracting and parabolic parameters  we will in fact see that $I(f_a)$ is  nowhere $G_{\delta\sigma}$, and $J(f_a)\setminus I(f_a)$ nowhere $F_{\sigma\delta}$. This latter is one of the conditions in a uniqueness theorem of van Engelen \cite{vee}. The other conditions involve Baire category and topological dimension, both of which are known for $J(f_a)\setminus I(f_a)$ \cite{bak,lip1}. In this way, Theorem \ref{t2} will follow from van Engelen's characterization.

\section{Preliminaries}

\subsection{Dynamics of entire functions}

For each positive integer $n$, the $n$-fold composition of an entire function $f$ is denoted $f^n$. The \textit{orbit} of a point $z\in \mathbb C$ is the sequence of iterates $(f^n(z))^\infty_{  {n=0}}$. 

A point $z$ belongs to the \textit{escaping set} $I(f)$ if $f^n(z)\to \infty$, that is, if the orbit of $z$ converges to the point at infinity on the Riemann sphere. 

The \textit{Julia set} $J(f)$ is the set of non-normality for the family of iterates of $f$. For $f_a$ this roughly means that  $z\in J(f_a)$ if every neighborhood of $z$ contains points whose orbits are very different from one-another; for instance,  points with periodic orbits and points  whose orbits go to $\infty$.  In context of  exponential functions  it is well-known that the Julia set is equal to the closure of the escaping set;  $J(f_a)=\overline{I(f_a)}$ \cite[Corollary 1]{ere}. 

\subsection{Attracting and parabolic parameters}The function $f_a$ has an attracting (or parabolic) cycle if there exists $z\in \mathbb C$ such that the orbit $(f_a^n(z))_{{  n=0}}^\infty$ is periodic and $|f_a'(z)|<1$ (or $|f_a'(z)|=1$). The number  $a$ is then called an attracting (parabolic) parameter. For example, $a=-2$ is attracting and  $a=-1$ is  parabolic.

The Julia set of $f_{-1}$ is a \textit{Cantor bouquet} of uncountably many disjoint rays (homeomorphic images of $[0,\infty)$); see \cite{dev,aa}. Each ray belongs to $I(f_{-1})$ with the possible exception of its finite endpoint  \cite[Theorem 4.2]{sch}.\footnote{    In many of the papers cited here, the functions $f_a$ with $a\in (-\infty,-1]$ are represented in the form of $\lambda e^z$ with $\lambda\in (0,\frac{1}{e}]$. Note that $f_a$ is conjugate to $e^a e^z$ via the translation $z\mapsto z-a$.}

More generally, if $a$ is any attracting or parabolic parameter then $I(f_{a})$ is a disjoint union of rays and curves (homeomorphic images of $(0,\infty)$). Different curves may terminate at the same point of $\mathbb C$. In the event that they do, the point at which they terminate is non-escaping and thus belongs to $J(f_a)\setminus I(f_a)$; see Figure 2.   This \textit{pinched Cantor bouquet} phenomenon was proved explicitly for attracting parameters in \cite{rem2}, and is explained further in  \cite[Section 3]{lip1}. Generalizations to larger classes of transcendental entire functions appear in \cite{alh}.      

\begin{figure}
\includegraphics[scale=0.6,trim={5cm 0 0 0},clip]{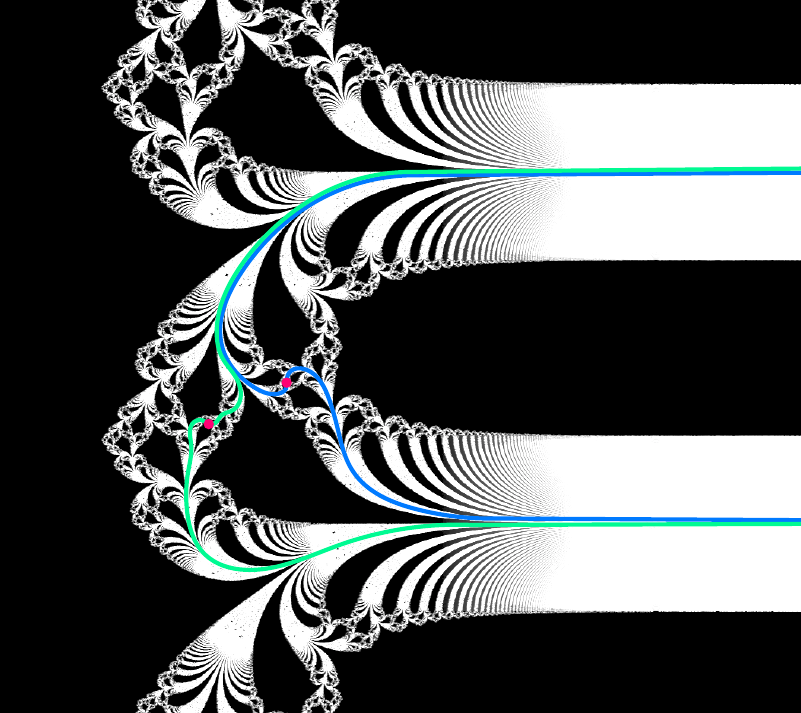}
\caption{Disjoint curves in $I(f_{2+\frac{\pi}{2}i})$ which terminate at non-escaping points.  See also \cite[Figure 1]{rem2}.}
\end{figure}

\subsection{Borel sets}All spaces under consideration are assumed to be separable and metrizable. A subset $X$ of a  space $Y$ is said to be an
\begin{itemize}
\item  $F_{\sigma}$-subset of $Y$ if $X$ is a countable union of closed subsets of $Y$
\item $G_{\delta}$-subset of $Y$ if $X$ is a countable intersection of open subsets of $Y$


\item $F_{\sigma\delta}$-subset of $Y$  if $X$ is a countable intersection of $F_{\sigma}$-subsets of $Y$
\item $G_{\delta\sigma}$-subset of $Y$  if $X$ is a countable union of $G_{\delta}$-subsets of $Y$.
\end{itemize}
Recall that in metric spaces, $G_{\delta}$-subsets include all closed subsets.

If  $Y$ is locally compact and $X$ is an $F_{\sigma}$-subset of $Y$, then $X$ is $\sigma$-compact. And then $X$ is an $F_{\sigma}$-subset of every  space into which it is embedded. In this event  $X$ is called an   absolute $F_{\sigma}$-set. 

If $Y$ is completely metrizable and $X$ is a $G_{\delta}$ (respectively, $F_{\sigma\delta}$ or $G_{\delta\sigma}$) subset of $Y$, then $X$ is a $G_{\delta}$  ($F_{\sigma\delta}$ or $G_{\delta\sigma}$) subset of every space $Z$ into which it is embedded.\footnote{    This is a consequence of  Lavrentiev's theorem \cite[Theorem 3.9]{kec} and the fact that $Z$ has a metric completion. See also \cite[p.432 Corollary 1 and Remark 1]{kur}} Now $X$ is called  an   absolute $G_{\delta}$-set (and similarly for $F_{\sigma\delta}$ and $G_{\delta\sigma}$). 

The absolute Borel properties described above are intrinsic to the space $X$ and are preserved by homeomorphisms.  Indeed, $X$ is an   absolute     :
\begin{itemize}
\item $F_{\sigma}$-set $\Leftrightarrow$ $X$ is  $\sigma$-compact
\item $G_{\delta}$-set $\Leftrightarrow$ $X$ is completely metrizable \cite[Theorem 3.11]{kec}
\item $F_{\sigma\delta}$-set $\Leftrightarrow$ $X$ has a Sierpi\'{n}ski stratification \cite{sie}
\item $G_{\delta\sigma}$-set $\Leftrightarrow$ $X$ is $\sigma$-complete (i.e.\ $X$ is a countable union of completely metrizable subspaces).
\end{itemize}

We say that a space $X$ is \textit{nowhere $G_{\delta\sigma}$} if {  every absolute $G_{\delta\sigma}$-subset of $X$ has empty interior}; likewise  for $F_{\sigma\delta}$.

\subsection{Baire category}A Borel set $X$ is \textit{first category} if $X$ can be written as a countable union of nowhere dense subsets, and \textit{Baire} if $X$ contains a dense completely metrizable subspace \cite[Theorem 1.12.2]{vee}. 

The first category property is inherited by open subspaces, and first category spaces are not Baire \cite[Theorem 8.4]{kec}. 

\subsection{Lower semi-continuity}A function $\varphi:X\to [0,\infty)$ is \textit{lower semi-continuous} if $\varphi(x_n)\to \varphi(x)$ whenever $x_n\to x$ and $\varphi(x_n)\leq \varphi(x)$ for all $n$. This is equivalent to $\varphi^{-1}(r,\infty)$ being open in $X$ for every $r\geq 0$.

\section{A topological model of $\exp-1$}

We will prove Theorem \ref{t1} first for the function $f=f_{-1}$ using a dynamical system $(J(\mathcal F),\mathcal F)$ which models $(J(f),f\restriction J(f))$.  The model was introduced in \cite{rem2} and was used extensively  in  \cite{rem}.

\subsection{The model}
Let $\mathbb Z ^\omega$ denote the space of  integer sequences $$\uline{s}=s_0s_1s_2\ldots.$$  Define 
$\mathcal F:[0,\infty)\times \mathbb Z ^\omega\to \mathbb R\times \mathbb Z ^\omega$ by $$\langle t,\uline{s}\rangle\mapsto \langle F(t)-|s_1|,\sigma(\uline{s})\rangle,$$ where $F(t)=e^t-1$ and $\sigma(s_0s_1s_2\ldots)=s_1s_2s_3\ldots$ is the shift  on $\mathbb Z ^\omega$.    
For each $x=\langle t,\uline{s}\rangle\in [0,\infty)\times \mathbb Z ^\omega$ put $T(x)=t$ and $\uline s(x)=\uline{s}$. Let  \begin{align*}
J(\mathcal F)&=\{x\in [0,\infty)\times \mathbb Z ^\omega:T(\mathcal F^n(x))\geq 0\text{ for all }n\geq 0\}\text{; and}\\  
 I(\mathcal F)&=\{x\in J(\mathcal F):T(\mathcal F^n(x))\to\infty\}.\end{align*}

\begin{ur}$J(\mathcal F)$  is closed in $[0,\infty)\times \mathbb Z ^\omega$ by continuity of each $T\circ \mathcal F^n$. \end{ur}

\begin{ur} In \cite[Section 9]{rem2} it was shown that $\mathcal F\restriction J(\mathcal F)$ is topologically conjugate to $f\restriction J(f)$, meaning that  there is a homeomorphism $\varphi:J(\mathcal F)\to J(f)$ such that $\varphi\circ \mathcal F\restriction J(\mathcal F)=f\circ \varphi.$   Technically, the conjugacy in  \cite{rem2} was constructed only in the case that  $f$ has an attracting cycle, whereas the fixed point of $f_{-1}$ is parabolic. However, $f_{-1}$ is conjugate on its Julia set to other exponentials with attracting fixed points, such as $f_{-2}$, due to more recent results in \cite{alh}.  See \cite[Example 11.1]{alh}.        \end{ur}

\subsection{Endpoints of $J(\mathcal F)$} 

For each $\uline s\in \mathbb Z^\omega$ put $$t_{\uline{s}}=\min\{t\geq 0:\langle t,\uline{s}\rangle\in J(\mathcal F)\},$$  or $t_{\uline{s}}=\infty$ if there is no such $t$. Observe that $$J(\mathcal F)=\bigcup_{\uline s\in \mathbb Z ^\omega} [t_{\uline s},\infty)\times \{\uline s\},$$ 
and 
 thus $E(\mathcal F)=\{\langle t_{\uline{s}},\uline{s}\rangle:t_{\uline{s}}<\infty\}$ consists of the (finite) endpoints of $J(\mathcal F)$. 

\begin{ur}In light of Remark 1 and the representation of $J(\mathcal F)$ above,  $$J(\mathcal F)\setminus E(\mathcal F)=\bigcup _{n=1}^\infty J(\mathcal F)+\langle 1/n,000\ldots \rangle$$ is an $F_{\sigma}$-subset of $J(\mathcal F)$. Hence $E(\mathcal F)$ is a $G_{\delta}$-subset of $J(\mathcal F)$  (and of the completely metrizable space $[0,\infty)\times \mathbb Z ^\omega$). Therefore $E(\mathcal F)$ is completely metrizable. \end{ur}

\begin{ur} The map $\uline s\mapsto t_{\uline s}$ is lower semi-continuous  \cite[Observation 3.1]{rem2}.\end{ur}

\begin{ur}{  The set of endpoints is completely invariant under the mapping $\mathcal F$.}  Moreover, for every $n\geq 0$ and $\langle t_{\uline{s}},\uline{s}\rangle\in E(\mathcal F)$ we have $$\mathcal F^n(\langle t_{\uline{s}},\uline{s}\rangle)=\langle t_{\sigma^n(\uline{s})},\sigma^n(\uline{s})\rangle\in E(\mathcal F).$$ In particular,  if $x\in E(\mathcal F)$ then $T(\mathcal F^n(x))=t_{\sigma^n(\uline s(x))}$.\end{ur}

\begin{ur}By Montel's theorem and the conjugacy between $\mathcal F\restriction J(\mathcal F)$ and $f\restriction J(f)$,  the completely $\mathcal F$-invariant sets $E(\mathcal F)$, $I(\mathcal F)$, and $E(\mathcal F)\cap I(\mathcal F)$ are each dense in $J(\mathcal F)$. \end{ur}

\subsection{Estimates of $t_{\underline s}$} 
 Let   $F^{-1}$ denote the inverse of $F$. So $F^{-1}(t)=\ln(t+1)$.  The $k$-fold composition of $F^{-1}$  will be denoted $F^{-k}$. The following can be verified with elementary calculus. 

\begin{up}
$F^{-k}(t-1)>F^{-k}(t)-1$ for all $k\geq 1$ and $t\in [1,\infty)$.\end{up}
 
\noindent Now for each $\uline s\in \mathbb Z ^\omega$ define $$t^*_{\uline s}=\sup_{k\geq 1}F^{-k}|s_{k}|.$$ 

\noindent The next proposition  comes from \cite[Lemma 3.8]{rem} and \cite[Observation 3.7]{rem}.

\begin{up}\label{p2}\
\begin{enumerate}[label=\textnormal{(\alph*)}]
\item $t^*_{\uline s}\leq t_{\uline s}\leq t^*_{\uline s}+1$,
\item  $\langle t_{\uline s},\uline s\rangle \in E(\mathcal F)\cap I(\mathcal F)$ if and only if  $t^*_{\uline s}<\infty$  and $t^*_{\sigma^n(\uline s)}\to\infty$, and
\item  if  $|s^0_n|\leq |s_n|$ for all $n<\omega$, then $ t_{\uline s^0}\leq t_{\uline s} $ (and likewise for $t^*$).
\end{enumerate}
\end{up}

\noindent We will also need the following.

\begin{up}\label{p3}For any positive real number $R>0$ and integer $n\geq 0$,  $$\{x\in J(\mathcal F):t^*_{\sigma^n(\uline s(x))}>R\}$$ is open in $J(\mathcal F)$. Further, if $y\in E(\mathcal F)\cap \overline{\{x\in E(\mathcal F):t^*_{\sigma^n(\uline s(x))}>R\}}$ then 
\begin{enumerate}[label=\textnormal{(\alph*)}]
\item $t_{\sigma^n(\uline s(y))}\geq R$, and 
\item $t^*_{\sigma^n(\uline s(y))}\geq R-1$.
\end{enumerate}\end{up}

\begin{proof}Let $\mathbb S=\{\uline s\in \mathbb Z ^\omega:t_{\uline s}<\infty\}$. Note that $\mathbb S$ is equal to $\{\uline s\in \mathbb Z ^\omega:t^*_{\uline s}<\infty\}$ by Proposition \ref{p2}(a). The mapping $\uline s\mapsto t^*_{\uline s}$ is easily seen to be lower semi-continuous, so $\mathbb U=\{\uline s\in \mathbb S:t^*_{\uline s}>R\}$  is open in $\mathbb S$. Thus $\sigma^{-n} \mathbb U=\{\uline s\in \mathbb S:\sigma^n(\uline s)\in \mathbb U\}$ is open in $\mathbb S$. By continuity of the projection of $J(\mathcal F)$ onto $\mathbb S$, we conclude that $$\{x\in J(\mathcal F):\uline s(x)\in \sigma^{-n}\mathbb U\}=\{x\in J(\mathcal F):t^*_{\sigma^n(\uline s(x))}>R\}$$ is open in $J(\mathcal F)$. 

 By Remark 5,  Proposition \ref{p2}(a)  and continuity of $T\circ \mathcal F^n$ we have 
\begin{align*}E(\mathcal F)\cap\overline{\{x\in E(\mathcal F):t^*_{\sigma^n(\uline s(x))}>R\}}&\subset E(\mathcal F)\cap\overline{\{x\in E(\mathcal F):t_{\sigma^n(\uline s(x))}>R\}}\\
&=E(\mathcal F)\cap\overline{\{x\in E(\mathcal F):T(\mathcal F^n(x))>R\}}\\
&\subset \{x\in E(\mathcal F):T(\mathcal F^n(x))\geq R\}\\
&=\{x\in E(\mathcal F):t_{\sigma^n(\uline s(x))}\geq R\}\\
&\subset \{x\in E(\mathcal F):t^*_{\sigma^n(\uline s(x))}+1\geq R\},
\end{align*}which proves both  (a) and (b).\end{proof}


\section{Stratifying the escaping endpoints}

Let $\mathbb N ^{<\omega}$ denote the set of all finite functions $\alpha:n\to \mathbb N$, where $n<\omega$ and $${  \dom(\alpha)=n=\{0,\ldots,n-1\}  }.$$ We sometimes represent $\alpha$ as an $n$-tuple of integers $\langle N_0,N_1,\ldots, N_{\dom(\alpha)-1}\rangle$ where $N_i=\alpha(i)$ for each $i<\dom(\alpha)$.  Given $\alpha\in \mathbb N^{<\omega}$ and $N\in \mathbb N$, the notation $\alpha^\frown N$ will stand for the extension of $\alpha$ that has representation $\langle N_0,N_1,\ldots, N_{\dom(\alpha)-1},N\rangle.$ For example, if $\alpha=\langle 1,2,5\rangle$ then $\alpha^\frown 8=\langle 1,2,5,8\rangle$.

We recursively define a system $(X_\alpha)$ of subsets of  $$\widetilde{E}(\mathcal F)=E(\mathcal F)\cap I(\mathcal F)=\{x\in E(\mathcal F):t_{\sigma^n(\uline s(x))}\to\infty\}.$$  
 for increasing functions $\alpha\in \mathbb N ^{<\omega}$.  To begin, let $$X_{\varnothing}=\widetilde E(\mathcal F).$$
  For each $N\in \mathbb N$ define $$X_{\langle N\rangle}=\{x\in X_{\varnothing}:t^*_{\sigma^n(\uline s(x))}> 2\text{ for all }n\geq N\}.$$ If $\alpha=\langle N_0,N_1,\ldots {  ,} N_{\dom(\alpha)-1}\rangle\in \mathbb N ^{<\omega}$ is increasing, $X_\alpha$ has been defined, and  $N>N_{\dom(\alpha)-1}$, then define 
$$X_{\alpha^\frown N}=\{x\in X_\alpha:t^*_{\sigma^n(\uline s(x))}> 3\dom(\alpha)+2\text{ for all }n\geq N\}.$$

\begin{obs}Every $X_{\alpha}$ is a $G_{\delta}$-subset of $\widetilde E(\mathcal F)$.\end{obs}

\begin{proof}This is an easy consequence of the first part of Proposition \ref{p3}. \end{proof}

\begin{obs}$$X_{\alpha}=\bigcup _{N=1}^\infty X_{\alpha^\frown N}$$ .\end{obs} 

\begin{proof}The inclusion ($\supset$) is  trivial, and ($\subset$) holds by Proposition \ref{p2}(b).\end{proof}

We will now show that every $X_{\alpha}$ is first category, as witnessed by the extensions $X_{\alpha^\frown N}$. In the proof below, the observation $$t^*_{\sigma^n(\uline s)}=\sup_{k\geq 1}F^{-k}|s_{n+k}|$$ will be helpful.

\begin{ut}\label{t3}$X_{\alpha^\frown N}$ is nowhere dense in $X_\alpha$. \end{ut}

\begin{proof}Let $\langle t_{\uline s},\uline s\rangle\in X_{\alpha^\frown N}$.  We will show that there is a sequence of points in $X_\alpha\setminus \overline{X_{\alpha^\frown N}}$ converging to $\langle t_{\uline s},\uline s\rangle$. To that end, for each $m<\omega$ define $\uline s^m$ coordinate-wise by  setting
$$
s^m_n=
\begin{cases}
s_n& \text{ if }n\leq m\\
\min\{|s_n|,\lfloor F^{n-m}(3\dom(\alpha))\rfloor\}& \text{ if }n>m.
\end{cases}
$$
Clearly $\uline s^m\to\uline s$ and $|s^m_n|\leq|s_n|$ for every $n$. So $t_{\uline s^m}\leq t_{\uline s}$ by Proposition \ref{p2}(c). From lower semi-continuity of $\uline s\mapsto t_{\uline s}$ we get   $\langle t_{\uline s^m},\uline s^m\rangle\to \langle  t_{\uline s},\uline s\rangle$.


  %

We will now prove  that a subsequence of $\langle t_{\uline s^m},\uline s^m\rangle$ is contained in $X_\alpha\setminus \overline{X_{\alpha^\frown N}}$. This will be established  by showing: 
\begin{center}\textit{For any  integer $M$ there exists $m\geq M$ such that $\langle t_{\uline s^m},\uline s^m\rangle\in X_\alpha\setminus \overline{X_{\alpha^\frown N}}$.}\end{center}

 For each $i<\dom(\alpha)$ put $N_i=\alpha(i)$,  and let $N_{\dom(\alpha)}=N$.  For  each $n\in [N_i,N_{i+1})$ there exists $k_n\geq 1$ such that $$F^{-k_n}|s_{n+k_n}|>3i+2.$$ 
   Now let $M$ be given; we may assume that $M>N+\max\{k_n:n<N\}$.      For each $n\in [N,M]$ there exists $k_n\geq 1$ such that $$F^{-k_n}|s_{n+k_n}|>3\dom(\alpha)+2.$$ Let $m=\max\{n+k_n:n\in [N,M]\}$. Clearly   $m\geq M+k_M>M$     . 

\begin{ucl}$\langle t_{\uline s^m},\uline s^m\rangle\in X_\alpha$\end{ucl}

\begin{proof}[Proof of Claim 1]  By the choice of $m$,       for every $i\leq \dom(\alpha)$ and $n\in [N_i,m)$ we have guaranteed that $t^*_{\sigma^n(\uline s^m)}>3\dom(\alpha\restriction i)+2.$ Additionally, if $n\geq m$ then $t^*_{\sigma^n(\uline s^m)}>3\dom(\alpha)-1.$  This is trivial if $s^m_{n'}=|s_{n'}|$ for all $n'>n$.  On the other hand, if there exists $n'>n$ such that $s^m_{n'}= \lfloor F^{n'-m}(3\dom(\alpha))\rfloor$ then by Proposition 1 we get    \begin{align*}
t^*_{\sigma^n(\uline s^m)}&\geq F^{-(n'-n)}\lfloor F^{n'-m}(3\dom(\alpha))\rfloor\\
&\geq F^{-(n'-n)}(F^{n'-m}(3\dom(\alpha))-1)\\ 
&> F^{-(n'-n)}(F^{n'-m}(3\dom(\alpha)))-1\\ 
&= F^{n-m}(3\dom(\alpha)))-1\\ 
&\geq  3\dom(\alpha)-1.
\end{align*} This also shows that $t^*_{\sigma^n(\uline s^m)}\to\infty$ in the case  when $$\{n<\omega: s^m_n=\lfloor F^{n-m}(3\dom(\alpha))\rfloor\}$$ is infinite, due to the fourth line in the inequality above. So in this case $\langle t_{\uline s^m},\uline s^m\rangle\in \widetilde E(\mathcal F)$ by Proposition 2(b). In the other case $\uline s^m$ is essentially $\uline s$, which clearly belongs to $\widetilde E(\mathcal F)$.   We thus have $\langle t_{\uline s^m},\uline s^m\rangle\in \widetilde E(\mathcal F)$, and conclude that $\langle t_{\uline s^m},\uline s^m\rangle\in X_\alpha$.      \end{proof}

\begin{ucl}$\langle t_{\uline s^m},\uline s^m\rangle\notin \overline{X_{\alpha^\frown N}}$\end{ucl}
\begin{proof}[Proof of Claim 2]Since $m\geq N$, the hypothesis  $\langle t_{\uline s^m},\uline s^m\rangle\in \overline{X_{\alpha^\frown N}}$ would imply $t^*_{\sigma^m(\uline s^m)}\geq 3\dom(\alpha)+1$ by Proposition \ref{p3}(b). But $$ t^*_{\sigma^m(\uline s^m)} =\sup_{k\geq 1} F^{-k}|s^m_{m+k}|\leq \sup_{k\geq 1} F^{-k}\lfloor F^{k}(3\dom(\alpha))\rfloor\leq 3\dom(\alpha).$$ Therefore $\langle t_{\uline s^m},\uline s^m\rangle\notin \overline{X_{\alpha^\frown N}}$.  \end{proof}

We have shown that each point of $X_{\alpha^\frown N}$ lies in the closure of $X_\alpha\setminus \overline{X_{\alpha^\frown N}}$. Therefore $X_{\alpha^\frown N}$ is nowhere dense in $X_\alpha$. This concludes the proof of Theorem \ref{t3}. \renewcommand{\qedsymbol}{$\blacksquare$}\end{proof}

\begin{uc}\label{c5}Each $X_{\alpha}$ is a first category space.\end{uc}

\begin{proof}Observation 2 and Theorem \ref{t3}.\end{proof}

\section{Proof of Theorem 1}
We are now ready for the main results.

\begin{ut}\label{t6}$\widetilde E(\mathcal F)$ is nowhere $G_{\delta\sigma}$. \end{ut}

\begin{proof}Let $d$ be a complete metric for $E(\mathcal F)$.  All closures in the proof will be taken in the space $E(\mathcal F)$, and diameters will be with respect to $d$.


Let $\mathcal A=\{A_n:n<\omega\}$ be a collection of completely metrizable subspaces of $\widetilde E(f)$. Let $W$ be any non-empty  open subset of $E(\mathcal F)$.  We will show that $\mathcal A$ does not cover $W\cap \widetilde E(\mathcal F)$. {  This will prove that $W$ is not $\sigma$-complete, and more generally that $E(\mathcal F)$ is nowhere $G_{\delta\sigma}$.}

{  Recall that any open subset of a first category space is again of first category.  } Hence the intersection $W\cap X_{\varnothing}$ is first category  by Corollary \ref{c5}. This set is also non-empty by Remark 6. Therefore $A_0$ is not dense in $W\cap X_{\varnothing}$. So  $W\cap X_{\varnothing}\setminus \overline {A_0}\neq\varnothing.$ Thus there is  a non-empty  open set $U\subset E(\mathcal F)$ such that $\overline U\subset W$, $\diam(U)<1$, and $\overline{U} \cap A_0=\varnothing$. Let $U_0=U\cap X_{\varnothing}$. 

By Observation 2 there exists $N_0$ such that  $X_{\langle N_0\rangle}\cap U_0\neq\varnothing$. Since $U_0\cap X_{\langle N_0\rangle}$ is first category (Corollary \ref{c5}),  it does not have a dense completely metrizable subspace.  By Observation 1,  $U_0\cap X_{\langle N_0\rangle}\cap A_1$  is a $G_{\delta}$-subset of $A_1$ and is therefore completely metrizable. So  $U_0\cap  X_{\langle N_0\rangle}\setminus \overline{U_0\cap  X_{\langle N_0\rangle}\cap A_1}\neq\varnothing.$ 
Hence there is a non-empty relatively open $U_1\subset U_0\cap  X_{\langle N_0\rangle}$ such that $\diam(U_1)<1/2$ and $\overline{U_1} \cap  A_1=\varnothing$.  Now choose  $N_1>N_0$ such that  $X_{\langle N_0,N_1\rangle}\cap U_1\neq\varnothing$. 

This process can be continued to get an increasing sequence $$\lambda=\langle N_0,N_1,N_2,\ldots\rangle\in \mathbb N ^\omega$$  and non-empty sets $U_0\supset U_1\supset U_2\supset\ldots$ such that $U_n$ is  open in $X_{\lambda\restriction n}$,  $\diam(U_{n})<\frac{1}{n+1}$, and $\overline{U_n}\cap A_n=\varnothing$. By completeness of the metric space $(E(\mathcal F),d)$  there exists $$x\in  \bigcap_{n=0}^\infty \overline{ U_n}$$  Then $x\in  \bigcap_{n=0}^\infty \overline{ X_{\lambda\restriction n}}$, so by   Proposition \ref{p3}(a) $t_{\sigma^n(\uline s(x))}\to\infty$.  We have  $$x\in\overline{U_0}\cap \widetilde E(\mathcal F)\subset W\cap \widetilde E(\mathcal F)$$ and yet $x\notin \bigcup \mathcal A$. Hence $\mathcal A$ does not cover $W\cap  \widetilde E(\mathcal F)$.\end{proof}

Recall that from Section 3 that $f=f_{-1}$, and $\mathcal F\restriction J(\mathcal F)$ is conjugate to $f\restriction J(f)$. 
\begin{uc}\label{c7}$I(f)$ is nowhere $G_{\delta\sigma}$. \end{uc}

\begin{proof}Since $I(f)$ is homeomorphic to $I(\mathcal F)$, it suffices to show that $I(\mathcal F)$ is nowhere a $G_{\delta\sigma}$-subset of $[0,\infty)\times \mathbb Z ^\omega$. This follows from Theorem \ref{t6}   and the fact that  $\widetilde E(\mathcal F)$ is a dense $G_{\delta}$-subset of $I(\mathcal F)$. \end{proof}

We can now  prove Theorem 1 by combining  Corollary 6 with the following.

\begin{up}[cf.\ {\cite[Theorem 1.1]{rem2}}]For every $a\in \mathbb C$ there exists $R>0$ and a homeomorphism $\varphi:\mathbb C\to \mathbb C$ such that $\varphi[A\cap I(f)]=\varphi[A]\cap I(f_a)$, where $$A=\{z\in \mathbb C:|f^n(z)|\geq R \text{ for all }n\geq 1\}.$$\end{up}

Let $a\in \mathbb C$ and let $R>0$ and $\varphi$ be given by Proposition 4.  Note that $$I(f)=\bigcup_{n=0}^\infty f^{-n}[A\cap I(f)].$$ Since $G_{\delta\sigma}$-sets are preserved by continuous pre-images and countable unions, by Corollary \ref{c7}   $A\cap I(f)$ is not $G_{\delta\sigma}$.  So $\varphi[A\cap I(f)]$ is not $G_{\delta\sigma}$.  This is a closed subset of $I(f_a)$ because it is equal to $\varphi[A]\cap I(f_a)$. Therefore  $I(f_a)$ cannot be  $G_{\delta\sigma}$. 
This concludes the proof of Theorem \ref{t1}.\hfill $\blacksquare$

\section{Proof of Theorem 3}

A space $X$ is \textit{zero-dimensional} if the topology of $X$ has a basis consisting of clopen subsets of $X$.  We  can now state van Engelen's  theorem.

\begin{up}[{\cite[Theorem A.2.6]{vee}}]Up to homeomorphism, there is only one zero-dimensional Baire space that is $G_{\delta\sigma}$ and nowhere $F_{\sigma\delta}$.  \end{up}

The standard representation of the space in Proposition 5 is $\mathbb R\setminus X$, where $X$ is a densely embedded copy of $\mathbb Q ^\omega$ (the infinite product of the rationals) in $\mathbb R$. It can be expressed more concretely as $\mathbb P^\omega\setminus (\mathbb Q+\pi)^\omega$, where $\mathbb P$ denotes the space of irrationals.

\begin{ut}If $f_a$ has an attracting or parabolic cycle, then $J(f_a)\setminus I(f_a)$ is a zero-dimensional Baire space that is $G_{\delta\sigma}$ and nowhere $F_{\sigma\delta}$.\end{ut}

\begin{proof}Clearly $J(f_a)\setminus I(f_a)$ is $G_{\delta\sigma}$ because its Julia complement $I(f_a)$ is $F_{\sigma\delta}$. It is zero-dimensional by \cite[Corollary 10]{lip1}, and Baire because it contains a dense $G_{\delta}$-set in the form of all points  whose orbits are dense in the Julia set \cite{bak,lip}.   Finally,  since $I(f_a)$ is dense in $J(f_a)$, an $F_{\sigma\delta}$-neighborhood in $J(f_a)\setminus I(f_a)$ would complement a $G_{\delta\sigma}$-neighborhood in $I(f_a)$. But $I(f_a)$ is nowhere $G_{\delta\sigma}$ by Corollary 7 and the equivalence  $I(f_a)\simeq I(f)$  (\cite[Theorem 1.2]{rem2} and \cite[Example 11.1]{alh}). Therefore $J(f_a)\setminus I(f_a)$ is nowhere $F_{\sigma\delta}$.\end{proof}

In light of Proposition 5, Theorem 8 implies Theorem 3.  \hfill $\blacksquare$

\end{document}